\documentclass[oneside,10pt,reqno]{amsart}
\usepackage{amssymb,amsmath,amsthm,bbm,enumerate,mdwlist,url,multirow,hyperref,amsthm,amsaddr}
\usepackage[pdftex]{graphicx}
\usepackage[shortlabels]{enumitem}
\usepackage[T1]{fontenc}
\allowdisplaybreaks
\theoremstyle{definition}
\newtheorem{definition}{Definition}
\theoremstyle{theorem}
\newtheorem{proposition}[definition]{Proposition}

\newtheorem{theorem}[definition]{Theorem}
\newtheorem{corollary}[definition]{Corollary}
\numberwithin{equation}{section}
\numberwithin{definition}{section}
\theoremstyle{remark}
\newtheorem{remark}[definition]{Remark}
\newtheorem{example}[definition]{Example}

\def\PP{\mathbb{P}}

\def\AA{\mathcal{A}}

\def\LL{\mathcal{L}}

\def\L2{\mathrm{L}^2}

\def\AA{\mathcal{A}}

\def\GG{\mathcal{G}}
\def\FF{\mathcal{F}}

\def\dd{\mathrm{d}}

 \def\L2{\mathrm{L}^2}

\def\dd{\mathrm{d}}

\makeatletter
\def\@tocline#1#2#3#4#5#6#7{\relax
  \ifnum #1>\c@tocdepth 
  \else
    \par \addpenalty\@secpenalty\addvspace{#2}%
    \begingroup \hyphenpenalty\@M
    \@ifempty{#4}{%
      \@tempdima\csname r@tocindent\number#1\endcsname\relax
    }{%
      \@tempdima#4\relax
    }%
    \parindent\z@ \leftskip#3\relax \advance\leftskip\@tempdima\relax
    \rightskip\@pnumwidth plus4em \parfillskip-\@pnumwidth
    #5\leavevmode\hskip-\@tempdima
      \ifcase #1
       \or\or \hskip 1em \or \hskip 2em \else \hskip 3em \fi%
      #6\nobreak\relax
    \hfill\hbox to\@pnumwidth{\@tocpagenum{#7}}\par
    \nobreak
    \endgroup
  \fi}
\makeatother

\makeatletter
\@namedef{subjclassname@2020}{%
  \textup{2020} Mathematics Subject Classification}
\makeatother
\begin{document}
\title[Continuous-state branching processes with migration]{Continuous-state branching processes with spectrally positive migration}

\author{Matija Vidmar}
\address{Department of Mathematics, Faculty of Mathematics and Physics, University of Ljubljana}
\email{matija.vidmar@fmf.uni-lj.si}

\begin{abstract}
Continuous-state branching processes (CSBPs) with immigration  (CBIs), stopped on hitting zero, are generalized by allowing the  process governing immigration  to be any  L\'evy process without negative jumps. 
Unlike CBIs, these newly introduced processes do not appear to satisfy any natural affine property on the level of the Laplace transforms of the semigroups.  Basic properties are noted. Explicit formulae (on neighborhoods of infinity) for the Laplace transforms of the first passage times downwards and  of the explosion time are derived. 
\end{abstract}

\thanks{The author is indebted to Cl\'ement Foucart for many stimulating discussions on the topic of this paper (and also for saving him from making an error in Corollary~\ref{corollary:hitting-zero}). Support from the Slovenian Research Agency  (project No. N1-0174) is acknowledged.} 

\date{\today}

\maketitle 
\section{Introduction}
\subsection{Motivation and mandate}
CSBPs (resp. CBIs) are the continuous analogues and scaling limits of the basic, but fundamental Bienaym\'e-Galton-Watson branching processes (resp. with independent constant-rate immigration). In \cite{Vidmar} there was added to the latter (so in discrete space) the phenomenon of  ``culling'', which is to say  emmigration (killing) of individuals at constant rate, but never more than one at any given point in time. One spoke of continuous-time Bienaym\'e-Galton-Watson processes with immigration and culling. Then \cite[Remark~2.1]{Vidmar} it was noted that these in turn should also allow for a continuous-space version. In this  article we construct said continuous-space analogues, christen them continuous-state branching processes with spectrallly positive migration (CBMs) [because unlike in the discrete-space case it is no longer possible (in general) to, as it were, separate out immigration and culling, so it is no longer appropriate to speak of the latter separately], derive their basic properties, finally we study their first-passage times downwards and explosion times (on the level of the Laplace transforms). The main results are as follows:  Theorem~\ref{prop:cbic-sde-method} for the construction and basic properties and Proposition~\ref{proposition:lamperti-one} for a Lamperti-style representation; Theorem~\ref{entrancetimeCBIculling}  for the first passage times and Theorem~\ref{thm:explosions} for explosions.

\subsection{General notation}\label{subsection:general-notation}
So as not to disturb the main line of the text too much we gather here some pieces of notation that we shall use frequently. Thus for a probability measure $\PP$ and a random element $Y$ defined thereunder: (i) $Y_\star\PP$ will be the law of $Y$ under the probability $\PP$, i.e. the probability measure $(A\mapsto \PP(Y\in A))$, its domain being understood from context; (ii) when $Y$ is numerical $\PP[Y]:=\int Y\dd \PP$ shall be its expectation under $\PP$; (iii) when further $\GG$ is a sub-$\sigma$-field, $\PP[Y\vert \GG]:=\mathbb{E}_\PP[Y\vert\GG]$ will designate the conditional expectation of $Y$ given $\GG$ under $\PP$. The symbols $\uparrow\uparrow$ and $\uparrow$ mean strictly increasing and nondecreasing, respectively; analogously for $\downarrow\downarrow$ and $\downarrow$.

\section{Construction of CBMs and first properties}
We are given two Laplace exponents of L\'evy processes having no negative jumps: \[\Psi_b(x):=\frac{\sigma^2_b}{2}x^2-\gamma_bx + \int \left(e^{-xh}-1+xh\mathbbm{1}_{(0,1]}(h)\right)\pi_{b}(\dd h),\quad x\in [0,\infty),\] 
where $\pi_b$ is a measure on $(0,\infty)$  satisfying $\int (1\wedge h^2) \pi_{b}(\dd h)<\infty$, $\sigma_b\in[ 0,\infty)$, $\gamma_b\in \mathbb{R}$;
 \[\Psi_m(x):=\frac{\sigma^2_m}{2}x^2-\gamma_mx + \int \left(e^{-xh}-1+xh\mathbbm{1}_{(0,1]}(h)\right)\pi_{m}(\dd h),\quad x\in [0,\infty),\] 
with the analogous qualifications on $(\pi_m,\sigma_m,\gamma_m)$. The subscripts $b$ and $m$ stand for branching and migration, respectively. The corresponding generators are given by
\begin{align*}
\mathcal{L}^{\Psi_b}f(z)&:=\frac{\sigma^2_b}{2}f''(z)+\gamma_b f'(z)+ \int_{0}^{\infty}\left(f(z+h)-f(z)-hf'(z)\mathbbm{1}_{(0,1]}(h)\right)\pi_b(\dd h)\\
\mathcal{L}^{\Psi_m}f(z)&:=\frac{\sigma^2_m}{2}f''(z)+\gamma_m f'(z)+ \int_{0}^{\infty}\left(f(z+h)-f(z)-hf'(z)\mathbbm{1}_{(0,1]}(h)\right)\pi_m(\dd h)
\end{align*}
for $f\in C^2_0(\mathbb{R})$ (i.e. for twice continuously differentiable $f:\mathbb{R}\to \mathbb{R}$ with $f,f',f''$ all vanishing at infinity), $z\in \mathbb{R}$; more generally $\mathcal{L}^{\Psi_b}f(z)$, $\mathcal{L}^{\Psi_m}f(z)$ are defined by the right-hand sides of the above display whenever the expressions appearing in them are defined. Let also  $\mu$ be a probability on the Borel sets of $[0,\infty)$, to be thought of as the initial distribution of the CBM. 

On a filtered probability space $(\Omega,\mathcal{H},\FF=(\FF_t)_{t\in [0,\infty)},\PP)$ satisfying the usual assumptions we prepare the following independent processes: an $\FF$-Poisson random measure  $\mathcal{M}_b(\dd s,\dd v,\dd h)$ on $[0,\infty)^3$ with intensity  $\dd s \dd v\pi_b(\dd h)$ (for Poisson random measures time will always be the first coordinate; also, when speaking of adaptedness to $\FF$, independent increments relative to $\FF$, etc., the measures are canonically identified with the associated point processes), two independent standard $\FF$-Brownian motions $W$ and $B$, an $\FF$-Poisson random measure $\mathcal{N}_m(\dd s,\dd h)$ on $[0,\infty)^2$ with intensity $\dd s\pi_m(\dd h)$. $\bar{ \mathcal{N}}_m$ and $\bar{\mathcal{M}}_b$ denote the compensated versions of $\mathcal{N}_m$ and $\mathcal{M}_b$, respectively. There is also an $\FF_0$-measurable random variable $X_0$ satisfying ${X_0}_\star\PP=\mu$. Remark that the sort of constellations described just now for sure exist for any given triplet $(\mu,\pi_b,\pi_m)$. Note also that automatically $(\mathcal{M}_b,W,B,\mathcal{N}_m)$ has in fact $\FF$-independent increments (jointly, not just each component separately) \cite[Theorem~II.6.3, Eq.~(II.6.12)]{ikeda1989stochastic}. 

We combine $X_0$, $\mathcal{N}_m$ and $B$ according to L\'evy-It\^o into the L\'evy process $X$ having no negative jumps, 
$$X_t=X_0+\sigma_m B_t+\gamma_m t+\int_{(0,t]\times [0,1]}h\bar{\mathcal{N}}_m(\dd s,\dd h)+\int_{(0,t]\times (1,\infty)}h\mathcal{N}_m(\dd s,\dd h),\quad t\in [0,\infty),\text{ a.s.-}\PP,$$
 in the filtration $\FF$, with Laplace exponent $\Psi_b$, and starting law $\mu$.  For $f$ and $z$ for which the right-hand side is defined  we also set
\begin{equation}
\mathcal{A}f(z):=\mathcal{L}^{\Psi_m}f(z)+z\mathcal{L}^{\Psi_b}f(z).
\end{equation}

Now the stage is set for the construction of the CBMs.

\begin{theorem}[SDE construction of CBMs]\label{prop:cbic-sde-method}
There exists a $\PP$-a.s. unique c\`adl\`ag, nonnegative real, $\FF$-adapted process  with lifetime $\zeta$, denoted $Y=(Y_t)_{t\in [0,\zeta)}$, having $0$ as an absorbing state, no negative jumps, and such that with $\tau_0$ the first entrance time into $\{0\}$ by $Y$, a.s.-$\PP$,
\begin{equation}\label{CBIcullingSDE}
Y_t=X_{t\land \tau_0}+\int_{(0,t]\times [0,Y_{s-}]\times [0,1]} \!\!\! \!\! \!\!\!\! \!\! \!\!\!\! \!\! \!\!\!\! \!\! \!\!\!\! \!\! \!\! \!\!h\bar{\mathcal{M}}_b(\dd s,\dd v,\dd h) +\int_{(0,t]\times [0,Y_{s-}]\times (1,\infty)} \!\!\! \!\! \!\!\!\! \!\! \!\!\!\! \!\! \!\!\!\! \!\! \!\!\!\! \!\! \!\! \!\!h{\mathcal{M}}_b(\dd s,\dd v,\dd h) +\sigma_b\int_{0}^{t}\sqrt{Y_{s}}\dd W_s+\gamma_b\int_{0}^tY_s\dd s
\end{equation}
for  $t\in [0,\zeta)$, also $\sup_{[0,\zeta)}Y=\infty$ a.s.-$\PP$ on $\{\zeta<\infty\}$ (implicitly, necessarily $\zeta>0$ a.s.-$\PP$).  The process $Y$ then further enjoys the following properties:
\begin{enumerate}[(i)]
\item\label{prop:cbic-sde-method:i}  it is adapted to the $\PP$-augmented natural filtration generated by $X_0$, $B$, $W$, $\mathcal{N}_m$, $\mathcal{M}_b$ (we mean of course $X_0$ as a constant process here; $Y$ is viewed as a process on $[0,\infty)$ by transferring it to the cemetery (natural: $\infty$) after $\zeta$) and its $\PP$-law, $\PP_\mu:=Y_\star \PP$ (completion implicit), is uniquely determined by the triplet $(\mu,\Psi_b,\Psi_m)$;
 \item\label{prop:cbic-sde-method:ii} it is quasi left-continuous on $[0,\zeta)$ in the filtration $\FF$, in the sense that for any sequence $(S_n)_{n\in \mathbb{N}}$ of $\FF$-stopping times that is $\uparrow$ a.s.-$\PP$ to some limit $S$ one has $\lim_{n\to\infty}Y(S_n)=Y(S)$ a.s.-$\PP$ on $\{S<\zeta\}$;
 \item\label{prop:cbic-sde-method:strongMarkov} it is strong Markov on $[0,\zeta)$ in the filtration $\FF$  in the sense that for any $\FF$-stopping time $S$, $\FF_S$ is $\PP$-independent of $Y_{S+\cdot}$ given $Y_S$ on $\{S<\zeta\}$, furthermore, for any nonnegative measurable map $G$, $\PP[G(Y_{S+\cdot})\vert S<\zeta]=\PP_\nu[G]$, where $\nu$ is the $\PP$-law of $Y_S$ conditionally on $\{S<\zeta\}$ (assuming of course $\PP(S<\zeta)>0$);
 \item\label{prop:cbic-sde-method:iv}  $\lim_{\zeta-}Y=\infty$ a.s.-$\PP$ on $\{0<\zeta<\infty\}$;
\item\label{prop:cbic-sde-method:mtg} 
\begin{enumerate}[(I)]
\item for any $\{\alpha,\bar\alpha\}\subset [0,\infty)$ and any $f\in C^2_c([0,\infty))$ (i.e. $f$ compactly supported and admitting a $C^2$ extension to a neigborhood of $[0,\infty)$) satisfying $\LL^{\Psi_m}f(0)=0$, the process $M$ given by 
\begin{align*}
M_t:=&f(Y_t)e^{-\alpha t-\bar\alpha\int_0^tY_s\dd s}\mathbbm{1}_{\{t<\zeta\}}-f(Y_0)\\
&-\int_0^{t\land \zeta}e^{-\alpha s-\bar\alpha\int_0^sY_u\dd u}\left(\mathcal{A}f(Y_s)-\alpha f(Y_s)-\bar\alpha Y_s f(Y_s)\right)\dd s,\quad  t\in [0,\infty),
\end{align*} 
is an $\FF$-martingale under $\PP$, vanishing at zero, that is bounded up to every deterministic time, so that in particular $\mathcal{A}$ is the generator of $Y$ on $\{g\in C^2_c([0,\infty)): \LL^{\Psi_m}g (0)=0\}$; 
\item\label{prop:cbic-sde-method:mtg:II}  if merely $f\in C^2([0,\infty))$ is bounded and $\LL^{\Psi_m}f(0)=0$, then the same process $M$ (restricted to $[0,\zeta)$) is a local martingale on $[0,\zeta)$ in $\FF$ under $\PP$, in the sense that there exists a sequence $(S_n)_{n\in \mathbb{N}}$ of $\FF$-stopping times that is $\uparrow\zeta$, each member of which is  $<\zeta$ a.s.-$\PP$ on $\{\zeta<\infty\}$, and such that $M^{S_n}=M_{S_n\land \cdot }$ is an $(\FF,\PP)$-martingale for each $n\in \mathbb{N}$.
\end{enumerate}
 \end{enumerate}
\end{theorem}
When the process $X$ is a subordinator (resp. the zero process), we recognize in \eqref{CBIcullingSDE} the stochastic equation solved by a CBI process \emph{stopped on hitting zero} with branching mechanism $\Psi_b$ and immigration mechanism $-\Psi_m$ (resp. a CSBP with branching mechanism $\Psi_b$). We add then here a further dynamics in that we allow $X$ to be any L\'evy process without negative jumps. Oscillations or negative drift  of the latter represent the possibility of culling, so that immigration is counterbalanced by killing/emmigrating an individual independently of the population size at constant rate in time. We stress that because of the presence of culling it is only natural, \emph{in general}, to stop the process $Y$ on hitting zero (unlike when $X$ is a (non-zero) subordinator), at least as long as we insist on the state space being $[0,\infty)$ (it is clearly visible if one considers the special case when $\gamma_b$, $\gamma_m$ are both $<0$, $\pi_c$ and $\pi_b$ are non-zero and carried by $(1,\infty)$, $\sigma_b=\sigma_m=0$). This is however not to a priori preclude the possibility that for other, but necessarily non-generic, constellations of input data the $Y$ of \eqref{CBIcullingSDE} could not  be naturally prolongated after $\tau_0$ as a nonnegative process (as happens for CBIs). 

We may also mention here that squared Bessel processes of negative dimension \cite[Definition~3]{bessel-negative} (stopped on hitting zero) are instances of CBMs. The SDE technique is not uncommon in  the branching literature, see e.g. \cite{palau-pardo,fekete_fontbona_kyprianou_2019,10.1214/ECP.v17-1972} and the references therein.\label{page:new-para}
\begin{proof}
Existence. Fix, for the time being, $n\in \mathbb{N}$. We claim that thanks to \cite[Theorems~3.2 and~5.1]{li-pu} there exists a pathwise unique strong (global, no lifetimes) solution $Y^n$ to the stochastic integral equation
\begin{align}\label{eq:sde-with-n}
Y_t^n=&Y^n_0+\int_0^t\left(\gamma_m+\gamma_b ((Y_{s-}^n\lor 0)\land n)\right)\dd s 
+\int_0^t\sigma_b\sqrt{(Y_{s-}^n\lor 0)\land n}\dd W_s+\int_0^t \sigma_m\dd B_s\\\nonumber
&+\int_{(0,t]\times U_0}g_0(Y_{s-}^n,u_0)\overline{N}_0(\dd s,\dd u_0)+\int_{(0,t]\times U_1}g_1(Y_{s-}^n,u_1)N_1(\dd s,\dd u_1),\quad t\in [0,\infty),\\\nonumber
Y_0^n=&X_0,
\end{align}
where  
\begin{align*}
U_0&:=([0,n]\times [0,1])\cup [0,1],\\
 U_1&:=([0,n]\times (1,\infty)) \cup (1,\infty),\\
 N_0(\dd s,\dd h)&:=\mathcal{N}_m(\dd s,\dd h)\text{ for }h\in [0,1],\\
 N_0(\dd s,\dd u_0)&:=\mathcal{M}_b(\dd s,\dd v,\dd h)\text{ for }u_0=(v,h)\in [0,n]\times [0,1],\\
 \overline{N}_0&=\text{ the compensated measure of $N_0$,}\\
 N_1(\dd s,\dd h)&:=\mathcal{N}_m(\dd s,\dd h)\text{ for }h\in (1,\infty),\\
  N_1(\dd s,\dd u_1)&:=\mathcal{M}_b(\dd s,\dd v,\dd h)\text{ for }u_1=(v,h)\in [0,n]\times (1,\infty),\\
   g_0(x,h)&:=h\text{ for }(x,h)\in \mathbb{R}\times [0,1], \\
   g_0(x,u_0)&:=\mathbbm{1}_{[v,\infty)}(x)h\text{ for }(x,u_0)=(x,(v,h))\in \mathbb{R}\times ([0,n]\times [0,1]),\\ 
   g_1(x,h)&:=h\text{ for }(x,h)\in \mathbb{R}\times (1,\infty)\text{ and }\\
   g_1(x,u_1)&:=\mathbbm{1}_{[v,\infty)}(x)h\text{ for }(x,u_1)=(x,(v,h))\in \mathbb{R}\times ([0,n]\times (1,\infty)). 
   \end{align*}
Indeed \eqref{eq:sde-with-n} is really just \eqref{CBIcullingSDE} except that one allows the process $Y^n$ to evolve after it hits negative values (without branching, just following $X$) and that the branching is also ``truncated'' at level $n$ (to preclude explosion), vis-\`a-vis the process $Y$. In fact, we may write more succinctly 
\begin{align} \label{eq:sde-with-n-succ}
Y_t^n=&X_{t}+\int_{(0,t]\times [0,(Y_{s-}^n\lor 0)\land n]\times [0,1]} \!\!\! \!\! \!\!\!\! \!\! \!\!\!\!\! \!\! \!\!\!\! \!\! \!\!\!\! \!\! \!\!\!\! \!\! \!\!\!\! \!\! \!\! \!\!h\bar{\mathcal{M}}_b(\dd s,\dd v,\dd h) \quad +\int_{(0,t]\times [0,(Y_{s-}^n\lor 0)\land n]\times (1,\infty)} \!\!\!\!\! \!\! \!\!\!\! \!\! \!\!\!\!\! \!\! \!\!\!\! \!\! \!\!\!\! \!\! \!\!\!\! \!\! \!\!\!\! \!\! \!\! \!\!h{\mathcal{M}}_b(\dd s,\dd v,\dd h)\\
&+\sigma_b\int_{0}^{t}\sqrt{(Y_{s}^n\lor 0)\land n}\dd W_s+\gamma_b\int_{0}^t(Y_{s}^n\lor 0)\land n\dd s,\nonumber\\
Y^n_0=& X_0.\nonumber
\end{align} However it is the form \eqref{eq:sde-with-n} given above, not \eqref{eq:sde-with-n-succ}, that relates directly to \cite[Eq.~(2.1)]{li-pu}.

Now, strictly speaking the stochastic equation for $Y^n$ does not fall under \cite{li-pu} because of the following minor point: in \cite[Eq.~(2.1)]{li-pu} the Brownian motion is one-dimensional -- in the preceding, two-variate. However, the generalization of the results of  \cite{li-pu} to a setting allowing a multidimensional Brownian motion is straightforward, especially if one of the two Brownian motions is just integrated against a constant (which is the present case). Furthermore, as far as the question of the existence of a strong pathwise unique solution is concerned, the integral $\int_{[0,t]\times U_1}$ may be ignored \cite[Proposition~2.1]{li-pu}. Once this has been noted it is easy to check that all the conditions of \cite{li-pu} required to establish the strong pathwise unique solution $Y^n$ of the above stochastic equation (trivialy adjusted to allow a two-variate Brownian motion) are in fact met. 

Let next $\zeta^n$ be the first time the stopped process $(Y^n)^{\tau^n_0}$ exits the interval $[0,n]$;  here $\tau^n_0$ is the first entrance time into $\{0\}$ by the process $Y^n$. Then $\zeta^{n+1}\geq \zeta^n$ and $Y^{n+1}=Y^n$ on $[0,\zeta^n)$ a.s.-$\PP$. Set $\zeta:=\lim_{n\to\infty}\zeta^n$ and $Y:=\lim_{n\to\infty} Y^n$ on $[0,\zeta)$ a.s.-$\PP$. We get all the properties stipulated for $Y$ in the ``unique existence'' part of the proposition. Thus existence is proved.

Uniqueness. Let $Y^I=(Y_t^I)_{t\in [0,\zeta^I)}$ and $Y^{II}=(Y^{II}_t)_{t\in [0,\zeta^{II})}$ both have the properties listed for $Y$ in the ``unique existence'' part of the proposition. Let $\tau^i_n$ (resp. $\tau_0^i$) be the first exit time from $[0,n]$ (resp. first entrance time into $\{0\}$) of $Y^i$, $n\in \mathbb{N}$, $i\in \{I,{II}\}$. Then, for each $i\in \{I,II\}$, $n\in \mathbb{N}$, the stopped process $(Y^i)^{\tau^i_n\land \tau_0^i}$ satisfies, on the interval $[0,\tau^i_n\land \tau_0^i]\cap [0,\zeta^i)$, the stochastic equation given above for the process $Y^n$. By the pathwise uniqueness of these solutions we obtain $Y^I=Y^{II}$ a.s.-$\PP$ on $[0,\tau^I_n\land \tau^{II}_n\land \tau_0^{I}\land \tau_0^{II}]\cap [0,\zeta^I\land \zeta^{II})$. Since $\sup_{[0,\zeta^i)}Y^i=\infty$ a.s.-$\PP$ on $\{\zeta^i<\infty\}$ and since $Y^i$ is c\`adl\`ag, therefore locally bounded on $[0,\zeta^i)$, we have that $\tau^i_n$ is $\uparrow$ $\zeta^i$ a.s.-$\PP$ as $n\to\infty$, $i\in \{I,II\}$. Therefore, passing to the limit $n\to\infty$, $Y^I=Y^{II}$ a.s.-$\PP$ on $[0,\tau_0^{I}\land \tau_0^{II}]\cap [0,\zeta^I\land \zeta^{II})$. Because $0$ is absorbing for $Y^i$, $i\in \{I,II\}$, it follows further that $Y^I=Y^{II}$ a.s.-$\PP$ on $[0,\zeta^I\land \zeta^{II})$. Next, write $I':=II$ and $II':=I$; because again $\sup_{[0,\zeta^i)}Y^i=\infty$ a.s.-$\PP$ on $\{\zeta^i<\infty\}$ and because $Y^{i'}$ admits left limits on $[0,\zeta^{i'})$ we get $\zeta^i\geq \zeta^{i'}$ a.s.-$\PP$ for each $i\in \{I,II\}$. In conclusion, $\zeta^I=\zeta^{II}$ a.s.-$\PP$ and finally $Y^1=Y^2$ a.s.-$\PP$. Therefore there is a.s.-$\PP$ unique existence.

As for the further properties of $Y$ we have as follows. 

\ref{prop:cbic-sde-method:i}. Adaptedness to the augmented natural filtration is by construction, for each $Y^n$, $n\in \mathbb{N}$, is a strong solution in its own right. That $\PP_\mu$ is uniquely determined by the triplet $(\mu,\Psi_b,\Psi_m)$ is because pathwise uniqueness implies uniqueness in law by a well-known general argument (the method is the same as in the proof of \cite[Theorem~IX.1.7]{MR1725357}). 

\ref{prop:cbic-sde-method:ii}. Quasi-left continuity is immediate from \eqref{CBIcullingSDE}. 

\ref{prop:cbic-sde-method:strongMarkov}. Let us infer the strong Markov property. On $\{S<\zeta\}$ introduce: $W':=(W_{S+t}-W_S)_{t\in [0,\infty)}$, the increments of $W$ after time $S$; $B':=(B_{S+t}-B_S)_{t\in [0,\infty)}$, the increments of $B$ after time $S$; $X':=Y_S+(X_{S+t}-X_S)_{t\in [0,\infty)}$, the increments of $X$ after $S$ offset by $Y_S$; $\mathcal{M}_b'(\dd s,\dd v,\dd h):=\mathcal{M}_b((\dd s-S)\cap (S,\infty),\dd v,\dd h)$, the shifted Poisson random measure $\mathcal{M}_b$; $\mathcal{N}_m'(\dd s,\dd h):=\mathcal{N}_m((\dd s-S)\cap (S,\infty),\dd h)$, the shifted Poisson random measure $\mathcal{N}_m$; finally $Y':=Y_{S+\cdot}$ with lifetime $\zeta':=\zeta-S$, the shifted process $Y$. Then we have, with $\tau_0'$  the first entrance time of $Y'$ into $\{0\}$, a.s.-$\PP$,
$$Y'_{t}=X'_{t\land \tau_0'}+\int_{(0,t]\times [0,Y_{s-}']\times [0,1]} \!\!\! \!\! \!\!\!\! \!\! \!\!\!\! \!\! \!\!\!\! \!\! \!\!\!\! \!\! \!\! \!\!h\bar{\mathcal{M}}_b'(\dd s,\dd v,\dd h) +\int_{(0,t]\times [0,Y'_{s-}]\times (1,\infty)} \!\!\! \!\! \!\!\!\! \!\! \!\!\!\! \!\! \!\!\!\! \!\! \!\!\!\! \!\! \!\! \!\!h{\mathcal{M}}_b'(\dd s,\dd v,\dd h) +\sigma_b\int_{0}^{t}\sqrt{Y_{s}'}\dd W'_s+\gamma_b\int_{0}^tY'_s\dd s$$ for $t\in [0,\zeta')$ on $\{S<\zeta\}$ (the bar in $\bar{\mathcal{M}}_b'$ designates the compensated measure); also, $0$ is absorbing for $Y'$, $Y'$ has no negative jumps and $\sup_{[0,\zeta')}Y'=\infty$ a.s.-$\PP$ on $\{\zeta'<\infty\}$. By the strong Markov property of $(\mathcal{M}_b,W,B,\mathcal{N}_m)$ in $\FF$ (applied to the stopping time $S'=S\mathbbm{1}_{\{S<\zeta\}}+\infty\mathbbm{1}_{\{\zeta\leq S\}}$; note $S'=S$ on $\{S'<\infty\}=\{S<\zeta\}$) and the very construction of $Y$ this means two things: first, that, $Y'$, being a measurable function of $(Y_S,\mathcal{M}_b',B',W',\mathcal{N}_m')$, is $\PP$-independent of $\FF_S$ given $Y_S$ on $\{S<\zeta\}$; second, ${Y'}_\star \PP(\cdot\vert S<\zeta)=\PP_\nu$ for $\nu:=(Y_S)_\star \PP(\cdot\vert S<\zeta)$. But this is precisely what we have stipulated under the strong Markov property.

\ref{prop:cbic-sde-method:iv}. If it is not the case that $\lim_{\zeta-}Y= \infty$ a.s.-$\PP$ on $\{0<\zeta<\infty\}$, then, for some $c\in (0,\infty)$,  with positive $\PP$-probability, $Y$ hits the level $c$ after having first gone above the level $c+1$, and does so consecutively infinitely many times over in finite time, which is in contradiction with the strong Markov property /coupled with the downwards skip-free property/ and the strong law of large numbers.

\ref{prop:cbic-sde-method:mtg}. The first martingale claim is a consequence of the second by bounded convergence. For the second martingale claim appeal, for each $n\in \mathbb{N}$, to It\^o's formula \cite[Theorem~II.5.1]{ikeda1989stochastic} for the stopped $(\FF,\PP)$-vector semimartingale $$\left(\left(Y_t^n,t,\int_0^tY_s^n\dd s\right)_{t\in[0,\infty)}\right)^{\zeta^n\land \tau_0^n}$$ to obtain $\PP$-a.s. for $t\in [0,\zeta^n\land \tau_0^n]\cap [0,\infty)$,
\begin{align*}
f(Y_t)e^{-\alpha t-\bar\alpha\int_0^tY_s\dd s}-f(Y_0)&=\int_0^{t}e^{-\alpha s-\bar\alpha\int_0^sY_u\dd u} f'(Y_s)(\sigma_m\dd B_s +\sigma_b\sqrt{Y_s}\dd W_s)\\
&\quad +\int_0^{t}e^{-\alpha s-\bar\alpha\int_0^sY_u\dd u}\left(f'(Y_s)(\gamma_m+\gamma_bY_s)-\alpha f(Y_s)-\bar\alpha Y_s f(Y_s)\right)\dd s\\
&\quad +\frac{1}{2}\int_0^te^{-\alpha s-\bar\alpha\int_0^sY_u\dd u}f''(Y_s)(\sigma_m^2\dd s+\sigma_b^2 Y_s\dd s)\\
&\quad +\int_{(0,t]\times (0,\infty)} \!\!\! \!\! \!\!\!\! \!\! \!\!\!\! \!\! e^{-\alpha s-\bar\alpha\int_0^sY_u\dd u}(f(Y_{s-}+h)-f(Y_{s-}))\bar{\mathcal{N}}_m(\dd s,\dd h)\\
&\quad +\int_0^t\int_{(0,\infty)} \!\!\!\! \!\!  e^{-\alpha s-\bar\alpha\int_0^sY_u\dd u}(f(Y_{s}+h)-f(Y_{s})-hf'(Y_s)\mathbbm{1}_{(0,1]}(h))\pi_m(\dd h)\dd s\\
&\quad +\int_{(0,t]\times [0,Y_{s-}]\times (0,\infty)} \!\!\! \!\! \!\!\!\! \!\! \!\!\!\! \!\! \!\!\!\! \!\! \!\!\!\! \!\! \!\! \!\!e^{-\alpha s-\bar\alpha\int_0^sY_u\dd u}(f(Y_{s-}+h)-f(Y_{s-}))\bar{\mathcal{M}}_b(\dd s,\dd v,\dd h)\\
&\quad + \int_0^t\int_{(0,\infty)} \!\!\!\! \!\!  e^{-\alpha s-\bar\alpha\int_0^sY_u\dd u}(f(Y_{s}+h)-f(Y_{s})-hf'(Y_s)\mathbbm{1}_{(0,1]}(h))Y_s\pi_b(\dd h)\dd s\\
&= \int_0^{t\land \zeta}e^{-\alpha s-\bar\alpha\int_0^sY_u\dd u}\left(\mathcal{A}f(Y_s)-\alpha f(Y_s)-\bar\alpha Y_s f(Y_s)\right)\dd s\\
&\quad +\int_0^{t}e^{-\alpha s-\bar\alpha\int_0^sY_u\dd u} f'(Y_s)(\sigma_m\dd B_s +\sigma_b\sqrt{Y_s}\dd W_s)\\
&\quad  +\int_{(0,t]\times (0,\infty)} \!\!\! \!\! \!\!\!\! \!\! \!\!\!\! \!\! e^{-\alpha s-\bar\alpha\int_0^sY_u\dd u}(f(Y_{s-}+h)-f(Y_{s-}))\bar{\mathcal{N}}_m(\dd s,\dd h)\\
&\quad +\int_{(0,t]\times [0,Y_{s-}]\times (0,\infty)} \!\!\! \!\! \!\!\!\! \!\! \!\!\!\! \!\! \!\!\!\! \!\! \!\!\!\! \!\! \!\! \!\!e^{-\alpha s-\bar\alpha\int_0^sY_u\dd u}(f(Y_{s-}+h)-f(Y_{s-}))\bar{\mathcal{M}}_b(\dd s,\dd v,\dd h).
\end{align*}
 On the r.h.s. of the preceding display we recognize local martingales in the last three lines (once we have stopped them at $\zeta^n\land \tau_0^n$). We deduce  that $M^{\zeta_n\land \tau_0^n}=M^{\zeta_n}$ (the equality thanks to $\mathcal{L}^{\Psi_m}f(0)=0$, which renders  $\AA f(0)=0$) is an $(\FF,\PP)$-local martingale for each $n\in \mathbb{N}$, so $M$ is ``locally a local martingale on $[0,\zeta)$''. By the usual trick (basically that of \cite[Lemma~1.35]{jacod1987limit}, mutatis mutandis to handle the lifetime $\zeta$) it follows that $M$ is a local martingale on $[0,\zeta)$ in the sense stipulated. 
\end{proof}
We call the (law of the) process $Y$ as rendered in the preceding theorem a continuous-state branching process with spectrally positive migration (CBM), \textit{branching} mechanism $\Psi_b$,  \emph{migration} one $\Psi_m$, initial law $\mu$. If needed, to emphasize $\mu$, we write $Y^\mu$ and/or $\PP^\mu$ in lieu of $Y$ and/or $\PP$, respectively (which should not be confused with $\PP_\mu=Y_\star \PP=(Y^\mu)_\star (\PP^\mu)$). For $a\in [0,\infty)$ we let $\tau_a$ be the first entrance time of $Y$ into $[0,a]$. \label{page:p-mu}

\begin{remark}
By a standard general argument  (cf. \cite[Theorem~IV.1.1]{ikeda1989stochastic}) one shows that the map $([0,\infty)\ni x\mapsto \PP_x(A))$ is universally measurable for each measurable $A$ and  $\PP[G(Y_{S+\cdot})\vert \FF_S]=\PP_{Y_S}[G]$ holds a.s.-$\PP$ on $\{S<\zeta\}$ in Theorem~\ref{prop:cbic-sde-method}\ref{prop:cbic-sde-method:strongMarkov}, also $\PP_\mu=\int \PP_z\overline{\mu}(\dd z)$. Here we have written (and will continue to write) $\PP_x:=\PP_{\delta_x}$, $x\in [0,\infty)$, for short. 
\end{remark}

\begin{remark}\label{remark:mtg}
In Theorem~\ref{prop:cbic-sde-method}\ref{prop:cbic-sde-method:mtg}, if for some $a\in [0,\infty)$ the initial law $\mu$ is carried by $[a,\infty)$ and if, ceteris paribus, the process $M$ is stopped at $\tau_a$, then, ceteris paribus, one can drop the assumption $\LL^{\Psi_m}f(0)=0$ and it is enough for the $C^2$ and boundedness property to prevail on $[a,\infty)$, and still the same martingale claims hold true. 
\end{remark}
In complete analogy with the discrete-space case \cite[Remark~2.1]{Vidmar}, the process $Y$ verifies  a random time-change integral equation involving two L\'evy processes having no negative jumps, one of which is $X$. 


\begin{proposition}[Lamperti transform for CBMs]\label{proposition:lamperti-one}
On an extension of the underlying probability space there exists a L\'evy process $L$ having no negative jumps, vanishing at zero a.s., with  Laplace exponent $\Psi_b$,  and such that 
\begin{equation}\label{eq:time-change}
Y_t=X_{t\land \tau_0}+L_{\int_0^tY_s\dd s},\quad t\in [0,\zeta),\text{ a.s.},
\end{equation}
also  [in what follows ${}^{-1}$ means the left-continuous inverse]
\begin{enumerate}[(a)]
\item\label{proposition:lamperti-one:a} ${L}$ has independent increments relative to the augmented natural filtration of the pair of processes $({L},(\int_0^{\cdot \land \zeta}{Y}_s\dd s)^{-1})$ initially enlarged by $\sigma({X})$, in particular $L$ is independent of $X$, while 
\item\label{proposition:lamperti-one:b} $X$ has independent increments relative to the augmented join of the natural filtration of $({L},(\int_0^{\cdot \land \zeta}{Y}_s\dd s)^{-1})$, time-changed by $\int_0^{\cdot \land \zeta}{Y}_s\dd s$, and of the natural filtration of $X$.
\end{enumerate}
\end{proposition}
\begin{proof}
We may and do assume $\FF$ is the augmented natural filtration of $(X_0,W,\mathcal{M}_b,B,\mathcal{N}_m)$. By first extending (if necesary) the probability space by an independent factor supporting a standard Brownian motion $H$ and a Poisson random measure $N(\dd u,\dd h)$ on $[0,\infty)^2$ with intensity $\dd u\pi_b(\dd h)$, we may and do also assume the latter were there to begin with. 

Put $\gamma_t:=\int_0^{t\land \zeta} Y_s\dd s$ for $t\in [0,\infty)$; then $\gamma\vert_{[0,\zeta\land \tau_0)}$ is $\uparrow\uparrow$, vanishing at zero and continuous. Let $\gamma^{-1}$ be its inverse, defined, $\uparrow\uparrow$, vanishing at zero and continuous on $[0,\rho)$, where $\rho:=\int_0^{\zeta}Y_s\dd s$. Additionally put $\gamma^{-1}:=\zeta\land \tau_0=\lim_{\rho-}\gamma^{-1}$ on $[\rho,\infty)$.  Thus $\gamma^{-1}$ is just the left-continuous inverse of $\gamma$ and $(\gamma^{-1}(u))_{u\in [0,\infty)}$ is a continuous $\uparrow$ family of $\FF$-stopping times, a time-change. Define the filtration $\GG$ as the augmented join of the time-changed filtration $\FF_{\gamma^{-1}}$, of the natural filtration of the pair of processes $(H,N)$, and of $\sigma(B,\mathcal{N}_m,X_0)$ (initial enlargement). Since $\{\rho\leq u\}=\{\gamma^{-1}(u)=\infty\}$ for all $u\in [0,\infty)$, we have that $\rho$ is a $\GG$-stopping  time. For $u\in [0,\rho)$ set  $\tilde W_u:=\int_0^{\gamma^{-1}(u)}\sqrt{Y_s}\dd W_s$ and $\tilde{\mathcal{N}}_b([0,u]\times A):=\int_0^{\gamma^{-1}(u)}\int_0^{Y_{s-}}\int_A \mathcal{M}_b(\dd s,\dd v,\dd h)$, $A\in \mathcal{B}_{[0,\infty)}$ (it is easy to check that the latter specifies uniquely a random measure on $[0,\rho)\times [0,\infty)$). 
 
By time-change (optional sampling) and independent enlargement we see that the process  $W'=\int_0^{\gamma^{-1}(\cdot)}\sqrt{Y}_s\dd W_s$, which agrees with $\tilde{W}$ on $[0,\rho)$, is a $\GG$-continuous local martingale vanishing at zero that is stopped at $\rho$ with terminal value $\tilde{W}_\rho:=\int_0^{\zeta}\sqrt{Y}_s\dd W_s$ a.s. on $\{\rho<\infty\}$. Its quadratic variation process is given by $\langle W'\rangle_u=\int_0^{\gamma^{-1}(u)}Y_s\dd W_s=u\land \rho$ a.s. for $u\in [0,\infty)$. Now define $\tilde{W}_u:=\tilde{W}_\rho+H(u)-H(\rho)$ for $u\in [\rho,\infty)$. Then we have that $\tilde{W}=W'+\mathbbm{1}_{[\![ \rho,\infty )\!)}(H-H(\rho))$ is a $\GG$-continuous local martingale vanishing at zero with increasing process given by $\langle \tilde{W}\rangle_u=u$ a.s. for $u\in [0,\infty)$. By L\'evy's martingale characterization  of Brownian motion \cite[Theorem~II.6.1]{ikeda1989stochastic} it follows that $\tilde{W}$ is a $\GG$-Brownian motion.

Similarly, again by time-change (optional sampling) and independent enlargement, we see that for each Borel set $A$ of $[0,\infty)$ of finite $\pi_b$-measure the process $\int_0^{\gamma^{-1}(\cdot)}\int_0^{Y_{s-}}\int_A \mathcal{M}_b(\dd s,\dd v,\dd h)-\pi_b(A)(\cdot\land \rho)$ is a $\GG$-martingale that is stopped at $\rho$ with terminal value $\int_0^{\zeta}\int_0^{Y_{s-}}\int_A \mathcal{M}_b(\dd s,\dd v,\dd h)-\pi_b(A)\rho=:\tilde{\mathcal{N}}_b([0,\rho]\times A)-\pi_b(A)\rho$ a.s. on $\{\rho<\infty\}$. Then define the random measure $\tilde{\mathcal{N}}_b$ on $[0,\infty)^2$ unambiguously by specifying further that $\tilde{\mathcal{N}}_b([0,u]\times A)=\tilde{\mathcal{N}}_b([0,\rho]\times A)+N((\rho,u]\times A)$ for $u\in [\rho,\infty)$ and $A\in \mathcal{B}_{[0,\infty)}$. We see that for each  Borel set $A$ of $[0,\infty)$ of finite $\pi_b$-measure the process $\tilde{\mathcal{N}}_b([0,\cdot]\times A)-\pi_b(A)\cdot$ is a $\GG$-martingale. It follows from the martingale characterization   of Poisson point processes \cite[Theorem~II.6.2]{ikeda1989stochastic} that $\mathcal{N}_b(\dd w,\dd h)$ is a $\GG$-Poisson random measure with intensity $\pi_b(\dd h)\dd w$.

Further, it is well-known that in a common filtration a Poisson point process and a Brownian motion are automatically independent \cite[Theorem~II.6.3]{ikeda1989stochastic}, not only that, they have jointly independent increments in said filtration, which is in fact what is proved in the quoted theorem, cf. \cite[Eq.~(II.6.12)]{ikeda1989stochastic}. Therefore, setting  (as usual a bar indicates the compensated measure)
$$L_u:=\sigma_b\tilde W_u+\gamma_bu+\int_{[0,u]\times [0,1]}\bar{\tilde{\mathcal{N}}}_b(\dd w,\dd h)+\int_{[0,u]\times (1,\infty)}{\tilde{\mathcal{N}}}_b(\dd w,\dd h),\quad u\in [0,\infty),\text{ a.s.},$$
we get a $\GG$-L\'evy process $L$ having no negative jumps with Laplace exponent $\Psi_b$. Moreover, from \eqref{CBIcullingSDE} we obtain
exactly \eqref{eq:time-change}. The fact that $L$ is a L\'evy process in the filtration $\GG$ gives \ref{proposition:lamperti-one:a}, just because $\GG$ contains the augmented natural filtration of the pair of processes $({L},\gamma^{-1})$ initially enlarged by $\sigma({X})$. On the other hand, $X$ has independent increments relative to $\FF$ initially enlarged by $\sigma(H,N)$ (and augmented), which contains the natural filtration of $X$ but also that of $({L},\gamma^{-1})$ time-changed by $\gamma$, since the latter is contained in $ (\FF_{\gamma^{-1}}\lor \sigma(H,N))_\gamma=(\FF_{\gamma^{-1}})_\gamma\lor \sigma(H,N)=\FF_{\cdot\land \tau_0\land\zeta}\lor \sigma(H,N)\subset \FF\lor \sigma(H,N)$ (for the second equality see \cite[Exercise~1.12]{MR1725357}, the first follows easily because $\sigma(H,N)$ is independent of $\FF_\infty$).
\end{proof}
Some historical comments on the preceding. When, ceteris paribus, $X=0$, the time-change delineated above is originally due to Lamperti \cite{MR0208685}, which explains the name ``Lamperti transform''; see also \cite{10.1214/09-PS154}. In fact, the Lamperti transform for CSBP works also in the other direction, constructing $Y$ from $L$.  \cite{10.1214/12-AOP766} generalizes the transform to CBIs (without stopping on hitting zero), albeit only in the latter direction,  starting from the pair $(X,L)$ to obtain $Y$. \cite{MA201511} handles the Lamperti transform of CSBP with competition (both ways) in the SDE setting;  many of the ideas of the proofs of Theorem~\ref{prop:cbic-sde-method} and of Proposition~\ref{proposition:lamperti-one} are from this source. 
 
It appears that obtaining the converse to Proposition~\ref{proposition:lamperti-one} is more involved for CBMs, vis-\`a-vis CBIs or CSBPs with competition, and this is left as an open problem. The main difficulty lies in establishing that \eqref{eq:time-change} has a unique solution for $Y$ given $(X,L)$, in the appropriate precise sense (if indeed it can be made precise in any reasonable way). Note that, on the one hand, the approach of \cite[p.~1603, proof of Theorem~1, uniqueness, esp. last display]{10.1214/12-AOP766} for CBIs relies heavily on the monotonicity of $X$, which is absent in the setting of CBMs. On the other hand, the line of attack of \cite[proof of Theorem~2.2]{MA201511} (to relate uniqueness of \eqref{eq:time-change} to the uniqueness of \eqref{CBIcullingSDE}) is hindered by the fact that, roughly speaking, one has to work simultaneously with the filtrations of $X$ and $L$ which ``run on different time-scales'' (it is not enough to just shift between time-changed filtrations, cf. Items~\ref{proposition:lamperti-one:a} and~\ref{proposition:lamperti-one:b} of Proposition~\ref{proposition:lamperti-one}).

We might also mention that a further Lamperti-sytle transform of a CBM process leads to ``CSBPs with collisions'', cf. e.g. \cite[Section~2, the process $R$ when $g$ is quadratic]{leman2019extinction} for a special case, however this connection will not be explored here. 

\begin{corollary}\label{corollary:explosivity-I}
Assume the CSBP with branching mechanism $\Psi_b$ is non-explosive. Then $\PP(\zeta<\infty)=0$, i.e. we have non-explosivity of the CBM as well.
\end{corollary}
This result will be refined to an equivalence in Corollary~\ref{corollary:explosivity-refined} (under assumption \eqref{eq:non-degnerate}).
\begin{proof}
Coupling argument. Suppose per absurdum that $\PP(\zeta<\infty)>0$. We may and do assume that $\mu=\delta_z$ for some $z\in (0,\infty)$. Let $\overline{X}=(\overline{X}_t)_{t\in[0,\infty)}$ be the running supremum of $X$: $\overline{X}_t:=\sup_{s\in[0,t]}X_s$ for $t\in [0,\infty)$. By continuity from below $\PP(\zeta<\infty,\overline{X}_\zeta< m)>0$ for some $m\in (z,\infty)$. Let $\tilde{Y}$ be the solution to \eqref{eq:time-change} with, ceteris paribus, $X\equiv m$ (for a constant $X$ we know that \eqref{eq:time-change} has an a.s. unique solution). So $\tilde{Y}$ is a CSBP with starting point $m$ and branching mechanism $\Psi_b$. All its corresponding quantities get a $\tilde{}$. We claim that $\tilde\zeta\leq \zeta$ a.s. on $\{\zeta<\infty,\overline{X}_\zeta< m\}$ (which implies  that $\tilde\zeta<\infty$ with positive probability, contradicting the non-explosivity of $\tilde{Y}$ and completing the proof). Suppose $\zeta<\tilde\zeta$ with positive probability on $\{\zeta<\infty,\overline{X}_\zeta< m\}$. Then we cannot have $\int_0^t Y_s\dd s\leq \int_0^t \tilde{Y}_s\dd s$ for all $t\in [0,\zeta)$ a.s. on $\{\zeta<\tilde\zeta,\overline{X}_\zeta< m\}$, since if it is true, letting $t\uparrow \zeta$ yields (by the Lamperti transform for $Y$) $\infty=\int_0^\zeta \tilde{Y}_s\dd s$ a.s. on the event $\{\zeta<\tilde\zeta,\overline{X}_\zeta< m\}$ of positive probability, contradicting the local boundedness of $\tilde{Y}$ on $[0,\tilde\zeta)$. So with positive probability on  $\{\zeta<\tilde\zeta,\overline{X}_\zeta< m\}$, $\int_0^t Y_s\dd s> \int_0^t \tilde{Y}_s\dd s$ for some $t\in [0,\zeta)$. Let $\delta:=\inf\{t\in [0,\zeta\land \tilde{\zeta}):\int_0^t Y_s\dd s> \int_0^t \tilde{Y}_s\dd s\}$. A.s. on $\{\delta<\zeta<\tilde{\zeta},\overline{X}_\zeta< m\}$, an event of positive probability, we have by the Lamperti transform
$$Y_\delta=X_\delta+L_{\int_0^\delta  Y_s\dd s}<m+L_{\int_0^\delta  Y_s\dd s}=m+L_{\int_0^\delta  \tilde Y_s\dd s}=\tilde{Y}_\delta$$
contradicting the fact that $\int_0^\cdot Y_s\dd s>\int_0^\cdot\tilde Y_s\dd s$ immediately after $\delta$.
\end{proof}

Let us conclude this section by emphasizing, at least on an informal level, the fundamental difference between CBI and CBM processes. The former are such that for independent $X_1$, $X_2$, $L_1$, $L_2$, $L$, with $L_1$ and $L_2$ having the same law as $L$, the process associated (via the Lamperti transform) to the pair $(X_1+X_2,L)$ has the same law as the sum of the processes associated to $(X_1,L_1)$ and $(X_2,L_2)$. Put more succinctly, CBIs can be superposed. This property fails for CBMs. Analytically it is a manifestation of the ``affine'' property of the Laplace transform of a CBI process \cite[Eq.~(1.1)]{KAW}, which cannot hold for the CBM class  \cite[Theorem~1.1]{KAW}. In this connection, one should emphasize that the exponential functions $e_\alpha:=e^{-\alpha\cdot}$, $\alpha\in (0,\infty)$, do not actually fall under Theorem~\ref{prop:cbic-sde-method}\ref{prop:cbic-sde-method:mtg} (not even the local martingale part, just because $\LL^{\Psi_m}e_\alpha(0)$ need not be $0$), at least not generically, so even though as a matter of analytical fact $\mathcal{A}_z e_\alpha(z)=\Psi_m(\alpha)e_\alpha(z)+z\Psi_b(\alpha)e_\alpha(z)=(\Psi_m(\alpha)-\Psi_b(\alpha)\frac{\partial}{\partial\alpha})e_\alpha(z)=:\mathcal{B}_\alpha e_\alpha(z)$, and even if $\Psi_m\geq 0$ (so that $\Psi_m$ may be interpreted as the instantaneous rate of killing in the operator $\mathcal{B}$), one is not able to simply ``integrate'' this duality on the level of (sic) the generators to a Laplace duality on the level of the semigroups  \cite[Proposition~1.2]{duality} (as is the case for CBI). Nevertheless:

\begin{proposition}\label{proposition:cbms-superposition}
Suppose two CBMs $Y^1$ and $Y^2$ have been prepared according to \eqref{CBIcullingSDE} using independent Brownian and Poisson drivers in the common filtration $\FF$ under a common probability $\PP$. All the quantities pertaining to $Y^i$ get a superscript $i$, $i\in \{1,2\}$. Suppose $\Psi_b^1=\Psi_b^2=:\Psi_b$. Then there exists (still on the same probability space, in the same filtration) a CBM process $Y=(Y_t)_{t\in [0,\zeta)}$ with initial distribution $\mu:=\mu^1\star\mu^2$, branching mechanism $\Psi_b$, migration mechanism $\Psi_m:=\Psi_m^1+\Psi_m^2$, such that, a.s.-$\PP$,  $\zeta\land \tau_0^1\land\tau_0^2= \zeta^1\land \zeta^2\land \tau_0^1\land\tau_0^2$ and $Y=Y^1+Y^2$ on $[0,\tau_0^1\land\tau_0^2\land \zeta)$. If $\Psi_m^2=0$ then we may further insist that, on the event $\{\tau_0^2<\tau_0^1\land \zeta\}$,  a.s.-$\PP$, $\zeta=\zeta^1$ and $Y=Y^1$ on $[\tau_0^2,\zeta)$.
\end{proposition}
\begin{proof}
Define  the random measure $\mathcal{M}_b(A):=\int_A \mathbbm{1}_{[0,Y_{s-}^1)}(v)\mathcal{M}_b^1(\dd s,\dd v,\dd h)+\int_A \mathbbm{1}_{[Y_{s-}^1,\infty)}(v)\mathbbm{1}_A(s,v-Y_{s-}^1,h)\mathcal{M}_b^2(\dd s,\dd v,\dd h)$, $A\in \mathcal{B}_{[0,\infty)^3}$, where we understand $Y^1=\infty$ on $[\zeta^1,\infty)$. Then \cite[Theorem~II.6.2]{ikeda1989stochastic} $\mathcal{M}_b(\dd s,\dd v,\dd h)$ is an $\FF$-Poisson random measure with intensity $\dd s\dd v\pi_b(\dd h)$. More trivially, $\mathcal{N}_m:=\mathcal{N}^1_m+\mathcal{N}^2_m$ is an $\FF$-Poisson random measure and the intensity of $\mathcal{N}_m(\dd s,\dd h)$ is $\dd s\pi_m(\dd h)$ with $\pi_m:=\pi_m^1+\pi_m^2$. Besides, $\PP$-a.s. $\mathcal{M}_b$ has no jumps in common with $\mathcal{N}_m$. 

Similarly, the process $W:=(Y^1+Y^2)^{-1/2}\cdot (\sqrt{Y^1}\cdot W^1+\sqrt{Y^2}\cdot W^2)$ defined on $[0,\zeta^1\land \zeta^2\land \tau_0^1\land \tau_0^2)$ and extended by the increments of $W^1$ thereafter, is a standard $\FF$-Brownian motion.  Again more trivially, $B:=((\sigma_m^1)^2+(\sigma_m^2)^2)^{-1}(\sigma_m^1 B^1+\sigma_m^2 B^2)$ (or just $B:=B^1$, say, if $\sigma_m^1=\sigma_m^2=0$) is a standard $\FF$-Brownian motion.  In addition, $\PP$-a.s. the covariation process of $W$ and $B$ vanishes.

From the preceding it follows \cite[Theorem~II.6.3]{ikeda1989stochastic} that the processes $\mathcal{M}_b$, $\mathcal{N}_m$, $B$ and $W$ are independent. Let $Y$ be the CBM corresponding to the initial value $X_0:=X_0^1+X_0^2$ and these drivers according to  \eqref{CBIcullingSDE}. Its branching mechanism is  $\Psi_b$, its migration mechanism is $\Psi_m$ and its initial value is $\mu$. Taking the sum of \eqref{CBIcullingSDE} corresponding to $Y^1$ and $Y^2$ we get the remainder of the claim (by the uniqueness of $Y$ as a solution to  \eqref{CBIcullingSDE}). 
\end{proof}

\begin{corollary}
CBMs are stochastically monotone in the starting law, with the (natural) convention that the coffin state is set equal to $\infty$, $[0,\infty]$ having the usual order: if $\mu'$ is another law on $\mathcal{B}_{[0,\infty)}$ with $\mu\leq \mu'$ (in first-order stochastic dominance) then $\mathbb{P}_\mu[G]\leq \mathbb{P}_{\mu'}[G]$ for any measurable map $G:[0,\infty]^{[0,\infty)}\to [-\infty,\infty]$ satisfying $\left(\omega\leq \omega'\Rightarrow G(\omega)\leq G(\omega')\right)$ for   $\{\omega,\omega'\}\subset [0,\infty]^{[0,\infty)}$ ($G$ is $\uparrow$ relative to the natural partial order induced on $[0,\infty]^{[0,\infty)}$ by the linear order $\leq$ on $[0,\infty]$) and $\PP_\mu[G^-]<\infty$.
\end{corollary}
\begin{proof}
Coupling. In the context of Proposition~\ref{proposition:cbms-superposition} take $X^2$ constant and equal to $X^2_0$ (so $\Psi^2_m=0$). $Y$ and $Y^1$ have the same branching and migration mechanisms but $Y$ starts above $Y^1$ a.s.-$\PP$. Also,  $\zeta\land \tau_0^1\land\tau_0^2= \zeta^1\land \zeta^2\land \tau_0^1\land\tau_0^2$ and $Y=Y^1+Y^2$ on $[0,\tau_0^1\land\tau_0^2\land \zeta)$; on the event $\{\tau_0^2<\tau_0^1\land \zeta\}$,  a.s.-$\PP$, $\zeta=\zeta^1$ and $Y=Y^1$ on $[\tau_0^2,\zeta)$. It follows that $Y\geq Y^1$ everywhere a.s.-$\PP$ [with the convention that the coffin state is set equal to $\infty$, $[0,\infty]$ having the usual order].
\end{proof}
\section{First passage times and explosions}\label{sec:first-explosion}
To avoid the analysis of some special cases, which are perhaps not of so much interest in the present context, we assume for the remainder of this text that (in the notation of the Lamperti transform, Proposition~\ref{proposition:lamperti-one}) 
\begin{center}
\fbox{\parbox{0.75\textwidth}{
\begin{equation}\label{eq:non-degnerate}
\text{neither $X$ nor $L$ have a.s. nondecreasing paths.}
\end{equation} 
}}
\end{center}
We call such CBMs non-degenerate. In other words, $\Psi_b$ and $\Psi_m$ are Laplace exponents of spectrally positive L\'evy processes (SPLPs)  [in the narrow sense] or of strictly negative drifts. Recall that the case when $X$ is a subordinator (resp. the zero process) corresponds to a CBI (resp. a CSBP) process and this has been studied in \cite{Duhalde} (at least as far as first passage times are concerned). So it is only the case when $L$ is a subordinator that is being left (completely) untreated.
 

Let $\xi$ be the canonical process on the space of c\`adl\`ag nonnegative real paths with lifetime; let $\mathsf{l}$ be the lifetime of $\xi$ and set $\sigma_a:=\inf\{t\in [0,\mathsf{l}): \xi_t\leq a\}$ for $a\in [0,\infty)$.  We trust that no confusion can arise when it comes to the notation $\sigma_a$ vis-\`a-vis the diffusion coefficients $\sigma_b$ and $\sigma_m$ (we will never use $b$ and $m$ for the first-passage level). 

Denote by $\Psi_b^{-1}:[0,\infty)\to [0,\infty)$ and $\Psi_m^{-1}:[0,\infty)\to [0,\infty)$ the right-continuous inverses of $\Psi_b$ and $\Psi_m$, respectively: $$\Psi_w^{-1}(t):=\inf\{z\in [0,\infty):\Psi_w(z)> t\},\quad t\in [0,\infty),\, w\in \{b,m\}$$ (finite, due to \eqref{eq:non-degnerate}). In the next theorem the case $\bar\alpha=0$ is mainly of interest, but the inclusion of $\bar\alpha>0$ does not really make the proof any longer or more difficult, and it allows to procure some information on the ``cumulative lifetime-to-date'' process $\int_0^\cdot \xi_s\dd s$.

\begin{theorem}[First passage times of CBMs]\label{entrancetimeCBIculling} Let $\alpha,\bar\alpha$ be from $[0,\infty)$. Suppose $\Psi_m(\Psi_b^{-1}(\bar \alpha))<\alpha$ (i.e.  $\Psi_b^{-1}(\bar \alpha)<\Psi_m^{-1}(\alpha)$ and [$\alpha>0$ or  $\Psi_b^{-1}(\bar \alpha)>0$]), or else suppose that $\Phi:=\Psi_b^{-1}(\bar \alpha)=\Psi_m^{-1}(\alpha)>0$. Put 
\begin{align}\label{eq:Phi}
\Phi_{\alpha,\bar\alpha}(x)&:=\int_{\Psi_b^{-1}(\bar{\alpha})}^\infty \frac{\dd z}{\Psi_b(z)-\bar{\alpha}}\exp\left(-xz-\int_{\Psi_m^{-1}(\alpha)}^z\frac{\Psi_m(u)-\alpha}{\Psi_b(u)-\bar\alpha}\dd u\right),\quad x\in [0,\infty),\\
&\text{or}\nonumber\\
\Phi_{\alpha,\bar\alpha}(x)&:=e^{-\Phi x}, \quad x\in [0,\infty),\nonumber
\end{align} 
 according as to whether $\Psi_m(\Psi_b^{-1}(\bar \alpha))<\alpha$ or $\Psi_b^{-1}(\bar \alpha)=\Psi_m^{-1}(\alpha)>0$.
Then for   $a\leq x$  from $[0,\infty)$,
\begin{equation}
\label{eq:firs-passage.cbic}
\PP_x[e^{-\alpha \sigma_a-\bar{\alpha} \int_{0}^{\sigma_a}\xi_s\dd s };\sigma_a<\mathsf{l}]=\frac{\Phi_{\alpha,\bar\alpha}(x)}{\Phi_{\alpha,\bar\alpha}(a)}.
\end{equation}
\end{theorem}
Before we proceed to the proof, some comments.
\begin{remark}\label{rmk:any-theta}
Any $\theta\in (\Psi_b^{-1}(\bar \alpha),\infty)$ may replace $\Psi_m^{-1}(\alpha)$ in \eqref{eq:Phi}, it changes it only by a multiplicative constant, which is immaterial. The delimiter $\Psi_m^{-1}(\alpha)$ seems most natural because it precisely separates the area of positivity and negativity of the integrand. Eq.~\eqref{eq:Phi} may be compared with the CBI case \cite[Eq.~(11)]{Duhalde}. It is somewhat agreeable that it actually attains a more ``symmetric'' form when viewed through the lense of Laplace exponents of SPLPs (as opposed to one SPLP and one subordinator).
\end{remark}

\begin{remark}
We see from \eqref{eq:Phi}-\eqref{eq:firs-passage.cbic}, by dominated convergence, that for all $a\in [0,\infty)$, $\lim_{b\to a}\tau_b=\tau_a$ a.s.-$\PP$. This may also be gleaned from the general properties of CBMs as follows. On the one hand, Proposition~\ref{proposition:lamperti-one} and the regularity downwards of SPLPs \cite[p.~232]{Kyprianoubook} imply that $Y$ is also regular downwards at all levels from $(0,\infty)$, which renders  $\text{$\downarrow$-$\lim$}_{b\uparrow a}\tau_b=\tau_a$ a.s.-$\PP$ for all $a\in (0,\infty)$. On the other hand, quasi left-continuity, coupled with the property $\lim_{\zeta-}Y=\infty$ a.s.-$\PP$ on $\{0<\zeta<\infty\}$, yield $\text{$\uparrow$-$\lim$}_{b\downarrow a}\tau_b=\tau_a$ a.s.-$\PP$ for all $a\in [0,\infty)$.
\end{remark}
As  a check:
\begin{example}\label{example:check}
Let $\{\mathsf{b},\mathsf{m}\}\subset (0,\infty)$ and $\Psi_b=\mathsf{b}\cdot\mathrm{id}_{[0,\infty)}$,  $\Psi_m=\mathsf{m}\cdot\mathrm{id}_{[0,\infty)}$. Then for $x\in [0,\infty)$: for  $\alpha\in (0,\infty)$, 
$\Phi_{\alpha,0}(x)=\frac{\Gamma(\frac{\alpha}{\mathsf{b}})}{\mathsf{b}\left(\frac{\alpha}{\mathsf{b}}+\frac{\alpha}{\mathsf{m}}x\right)^{\frac{\alpha}{\mathsf{b}}}},$ hence $\PP_x[e^{-\alpha \sigma_0};\sigma_0<\mathsf{l}]=(1+\frac{\mathsf{b}}{\mathsf{m}}x)^{-\frac{\alpha}{\mathsf{b}}}$; therefore ${\sigma_0}_\star \mathbb{P}_x=\delta_{\mathsf{b}^{-1}\log(1+\frac{\mathsf{b}}{\mathsf{m}}x)}$, as it should be (e.g. from \eqref{eq:time-change}).
\end{example}

Formula \eqref{eq:Phi} above, and \eqref{eq:scale-for-explosion} to follow below, may seem at first sight to appear ``out of the blue''. Of course it is not so. First, they may be guessed from the discrete counterparts \cite[Corollary~4.14, Theorem~4.2]{Vidmar}, using \cite[Remark~4.16]{Vidmar}, as was actually the case. Second, in absence of the former, they could be got by solving the relevant o.d.e. problems (though it still involves ``guessing'' the ``Laplace transform/completely monotone'' forms of  \eqref{eq:Phi} \& \eqref{eq:scale-for-explosion}, cf. the discrete case \cite[p.~9]{Vidmar}). With the luxury of the discrete analogs being available, the first option seems decidedly preferable (or anyway faster/easier). 

\begin{proof}
We focus on the case when $\Psi_m(\Psi_b^{-1}(\bar \alpha))<\alpha$, the version with $\Psi_b^{-1}(\bar \alpha)=\Psi_m^{-1}(\alpha)>0$ is similar (also easier) and is left to the reader. 

First, $\Phi_{\alpha,\bar\alpha}:[0,\infty)\to (0,\infty)$ is well-defined, finite, $\downarrow\downarrow$, continuous and vanishing at infinity, which is easy to see using the assumption $\Psi_m(\Psi_b^{-1}(\bar \alpha))<\alpha$  by recalling that $\Psi_m-\alpha$ is continuous with $\lim_\infty\Psi_m=\infty$, also strictly negative on $(0,\Psi_m^{-1}(\alpha))$ and strictly positive on $(\Psi_m^{-1}(\alpha),\infty)$, with the analogous observation being true for $\Psi_b$. Differentiating under the integral sign we see also that $\Phi_{\alpha,\bar\alpha}$ is $C^2$ (indeed $C^\infty$) on $(0,\infty)$ (maybe or maybe not at zero, we do not need it). 

We let $\mu=\delta_x$ be the initial distribution of the CBM $Y$, thus $\PP_x$ is the $\PP=\PP^\mu$-law of $Y=Y^\mu$. If necessary, taking the limit as $\alpha\downarrow 0$, we see by monotone (or bounded) convergence on the l.h.s. and by monotone and dominated convergence on the r.h.s. of \eqref{eq:firs-passage.cbic}, that we may (and we do) assume $\alpha>0$ (indeed, for $\alpha=0$, one can first replace the delimiter $\Psi_m^{-1}(\alpha)$ with $\Psi_m^{-1}(0)$ in the expression for $\Phi_{\alpha,\bar\alpha}$, then pass to the limit $\alpha\downarrow 0$ by monotone convergence on $z\in (\Psi_b^{-1}(\bar\alpha),\Psi_m^{-1}(0))$ and by dominated convergence on $z\in (\Psi_m^{-1}(0),\infty)$). Because $Y$ is quasi left-continuous, due to the continuity of $\Phi_{\alpha,\bar\alpha}$ and because now $\alpha>0$, we further infer that it suffices to establish the Laplace transform formula for $a>0$ (one can pass to the limit $a\downarrow 0$ by bounded convergence on the l.h.s. and by continuity of $\Phi_{\alpha,\bar\alpha}$ on the r.h.s. of \eqref{eq:firs-passage.cbic}). Thus we may and do assume $a>0$.

Consider now the process $M$ of Theorem~\ref{prop:cbic-sde-method}\ref{prop:cbic-sde-method:mtg}\ref{prop:cbic-sde-method:mtg:II} with $f=\Phi_{\alpha,\bar\alpha}$, but stopped at $\tau_a$, i.e., modulo its deterministic initial value, the process \footnotesize
$$\left(\Phi_{\alpha,\bar\alpha}(Y_{t\land \tau_a})e^{-\alpha (t\land \tau_a)-\bar\alpha\int_0^{t\land \tau_a}Y_s\dd s}\mathbbm{1}_{\{t\land\tau_a<\zeta\}}-\int_0^{t\land \tau_a\land \zeta}\!\!\!\!e^{-\alpha s-\bar\alpha\int_0^sY_u\dd u}\left(\mathcal{A}\Phi_{\alpha,\bar\alpha}-\alpha \Phi_{\alpha,\bar\alpha}-\bar\alpha \mathrm{id}_{[0,\infty)} \Phi_{\alpha,\bar\alpha}\right)(Y_s)\dd s\right)_{t\in [0,\infty)}.$$\normalsize
 Notice that  $Y_{\tau_a}=a$  on $\{\tau_a<\zeta\}=\{\tau_a<\infty\}$ (since $Y$ has no negative jumps), while $\lim_{t\uparrow \zeta}\Phi_{\alpha,\bar\alpha}(Y_t)e^{-\alpha t-\bar\alpha\int_0^tY_s\dd s}=0$ a.s.-$\PP$ on $\{\tau_a=\infty\}$ (since $\alpha>0$ and $\Phi_{\alpha,\bar\alpha}$ vanishes at infinity). Therefore, by Theorem~\ref{prop:cbic-sde-method}\ref{prop:cbic-sde-method:mtg}\ref{prop:cbic-sde-method:mtg:II} \& Remark~\ref{remark:mtg}, and because martingales have a constant expectation, the Laplace transform formula \eqref{eq:firs-passage.cbic} reduces further to establishing that $\Phi_{\alpha,\bar\alpha}$ is $C^2$, bounded on $[a,\infty)$ and satisfies
 $$\mathcal{A}\Phi_{\alpha,\bar\alpha}(x)=\alpha\Phi_{\alpha,\bar\alpha}(x)+\bar \alpha x\Phi_{\alpha,\bar\alpha}(x),\quad x\in [a,\infty).$$



This however is a straightforward computation made easy by the fact that we are working on restriction to $x\in [a,\infty)\subset(0,\infty)$ (justifying differentiation under the integral sign): 
\begin{align*}
\mathcal{A}\Phi_{\alpha,\bar\alpha}(x)=&\mathcal{L}^{\Psi_m}\Phi_{\alpha,\bar\alpha}(x)+x\mathcal{L}^{\Psi_b}\Phi_{\alpha,\bar\alpha}(x)\\
=&\int_{\Psi_b^{-1}(\bar{\alpha})}^\infty \frac{\dd z}{\Psi_b(z)-\bar{\alpha}}\exp\left(-xz-\int_{\Psi_m^{-1}(\alpha)}^z\frac{\Psi_m(u)-\alpha}{\Psi_b(u)-\bar\alpha}\dd u\right)\left(\Psi_m(z)+x\Psi_b(z)\right)\\
=&\int_{\Psi_b^{-1}(\bar{\alpha})}^\infty \frac{\dd z}{\Psi_b(z)-\bar{\alpha}}\exp\left(-xz-\int_{\Psi_m^{-1}(\alpha)}^z\frac{\Psi_m(u)-\alpha}{\Psi_b(u)-\bar\alpha}\dd u\right)\\
&\qquad\qquad\qquad\times \left(\Psi_m(z)-\alpha+\alpha+x(\Psi_b(z)-\bar\alpha+\bar\alpha)\right)\\
=&\alpha\Phi_{\alpha,\bar\alpha}(x)+\bar \alpha x\Phi_{\alpha,\bar\alpha}(x)\\
&+\int_{\Psi_b^{-1}(\bar{\alpha})}^\infty\dd z \frac{\Psi_m(z)-\alpha}{\Psi_b(z)-\bar{\alpha}}\exp\left(-xz-\int_{\Psi_m^{-1}(\alpha)}^z\frac{\Psi_m(u)-\alpha}{\Psi_b(u)-\bar\alpha}\dd u\right)\\
&+x\int_{\Psi_b^{-1}(\bar{\alpha})}^\infty\dd z\exp\left(-xz-\int_{\Psi_m^{-1}(\alpha)}^z\frac{\Psi_m(u)-\alpha}{\Psi_b(u)-\bar\alpha}\dd u\right)\\
=&\alpha\Phi_{\alpha,\bar\alpha}(x)+\bar \alpha x\Phi_{\alpha,\bar\alpha}(x)-\left[\exp\left(-xz-\int_{\Psi_m^{-1}(\alpha)}^z\frac{\Psi_m(u)-\alpha}{\Psi_b(u)-\bar\alpha}\dd u\right)\big \vert_{z=\Psi_b^{-1}(\bar\alpha)}^\infty\right]\\
=&\alpha\Phi_{\alpha,\bar\alpha}(x)+\bar \alpha x\Phi_{\alpha,\bar\alpha}(x),
\end{align*}
where the penultimate equality is integration by parts and the last equality is by elementary estimation using $\Psi_m(\Psi_b^{-1}(\bar\alpha))<\alpha$ (at $z=\infty$ it is trivial, at $z=\Psi_b^{-1}(\bar\alpha)+$ one gets a divergent integral of the form $\sim\int_{0+}\frac{\dd u}{u}$).
\end{proof}

\begin{corollary}\label{corollary:convergence-of-CBM}
We have that $\PP_x(\sigma_0<\infty)>0$ for all $x\in [0,\infty)$ and $\lim_{\mathsf{l}-}\xi=\infty$ a.s. on $\{\sigma_0=\infty\}$, in particular there is no phenomenon of extinguishing (i.e. the event $\{\lim_\infty \xi=0\}\cap \{\sigma_0=\infty\}$ is negligible).
\end{corollary}
\begin{remark}
We may recall that CSBP can extinguish \cite[p.~343]{Kyprianoubook}.
\end{remark}
\begin{proof}
It is clear from the strict positivity of the scale function of \eqref{eq:Phi} that $\PP_x(\sigma_0<\infty)>0$ for all $x\in [0,\infty)$. 

Suppose per absurdum that $\lim_{\mathsf{\zeta}-}Y\ne \infty$ with positive $\PP$-probability on $\{\tau_0=\infty\}$. Then for some $N\in [0,\infty)$, on an event $A$ of positive $\PP$-probability, the process $Y$ will be $\leq N$ at arbitrarily large times, but never hit zero. 

Fix some $\alpha\in (0,\infty)$. Consider the sequence $(S_k)_{k\in \mathbb{N}}$ of random times defined as follows. First, if needed, enlarge the probability space to gain access to  $e_\alpha^{(k)}$, $k\in \mathbb{N}$, independent exponentially with rate $\alpha$ distributed $(0,\infty)$-valued random variables, independent of $Y$. Second, put, inductively,
$$
S_{k}:=\inf\{t\in [S_{k-1},\zeta):Y_t\in [0,N]\}+e_\alpha^{(k)},\quad k\in \mathbb{N},
$$
with the convention $S_0:=0$. Thus, in plainer tongue, $S_1=(\text{the first time $Y$ enters $[0,N]$})+e_\alpha^{(1)}$; $S_2:=(\text{the first time $Y$ enters $[0,N]$ after $S_1$})+e_\alpha^{(2)}$; and so on. Perhaps $S_k=\infty$ at some $k\in \mathbb{N}$, in which case $S_{l+1}=\infty$ for all $l\in \mathbb{N}$, $l\geq k$.  But anyway $A\subset \{S_k<\infty\text{ for all }k\in \mathbb{N}\}$. Now, $Y$ always has a strictly positive chance $\beta>0$ to hit zero before an independent exponential random time of rate $\alpha$ has elapsed, no matter where in $[0,N]$ it starts. On $A$ it must fail to do so infinitely many times over. By the strong Markov property it is impossible. (Of course a sufficiently large deterministic time could also be used in lieu of the $e_\alpha^{(k)}$, $k\in \mathbb{N}$.)
\end{proof}

\begin{corollary}\label{corollary:hitting-zero}
Suppose  $(\Psi_b)'(0+)\geq 0$ (i.e. $\Psi_b^{-1}(0)=0$) and fix $\theta\in (0,\infty)$. Then $\PP_x(\sigma_0<\infty)=1$ for all $x\in [0,\infty)$ iff 
\begin{equation}\label{eq:abelian}
\int_0^{\theta}\frac{\dd z}{\Psi_b(z)}\exp\left(\int_z^{\theta}\frac{\Psi_m(u)}{\Psi_b(u)}\dd u\right)=\infty
\end{equation}
(which is true if $(\Psi_m)'(0+)>-\infty$ and $(\Psi_b)'(0+)>0$);
 when \eqref{eq:abelian} fails, then 
 \begin{equation}\label{eq:hitting-zero}
 \PP_x(\sigma_0<\infty)=\frac{\int_0^\infty \frac{\dd z}{\Psi_b(z)}e^{-xz-\int_{\theta}^z\frac{\Psi_m(u)}{\Psi_b(u)}\dd u}}{\int_0^\infty \frac{\dd z}{\Psi_b(z)}e^{-\int_{\theta}^z\frac{\Psi_m(u)}{\Psi_b(u)}\dd u}}<1\text{ for all }x\in (0,\infty).
 \end{equation}
\end{corollary}
Thus when $(\Psi_b)'(0+)> 0$ and $(\Psi_m)'(0+)>-\infty$ the migrations cannot offset the branching to turn the process from one that a.s. becomes extinct or extinguished to one for which this would not be the case, while the situation for $(\Psi_b)'(0+)>0$ but $(\Psi_m)'(0+)=-\infty$ appears to be  more delicate:
\begin{example}\label{example:extinction}
Let $\Psi_b=\mathrm{id}_{[0,\infty)}$ ($\therefore$ $\Psi_b'(0+)>0$), $\sigma_m^2>0$, $\gamma_m=0$, $\pi_m$ vanishing on $(0,2]$. If $\pi_m((z,\infty))=2/\log(z)$ for $z\in [2,\infty)$ ($\therefore$ $(\Psi_m)'(0+)=-\infty$; such a constellation \emph{can} transpire), then it is elementary (if tedious) to check using the abelian theorem for Laplace transforms \cite[Theorem~XIII.5.2]{feller} that the integral of \eqref{eq:abelian} converges. On the other hand if $\pi_m((z,\infty))=z^{-1/2}$ for $z\in [2,\infty)$  ($\therefore$ $(\Psi_m)'(0+)=-\infty$; and again such a constellation \emph{can} transpire), then one checks similarly that the integral of \eqref{eq:abelian} diverges.
\end{example}
Condition \eqref{eq:abelian} may also be compared with the recurrence/transience condition of not-supercritical  CBIs \cite[Theorem~3(a)]{Duhalde}. Examples in which $\Psi_b'(0+)=0$, $\Psi_m'(0+)>-\infty$ and when \eqref{eq:abelian} fails or obtains can be reconstructed from those  of \cite[Corollary~4]{Duhalde} (say by adding a Brownian component to the $\Phi$ featuring there).
\begin{proof}
Take $\bar\alpha=0$, $\alpha>0$, $a=0$ in \eqref{eq:firs-passage.cbic} and pass to the limit $\alpha\downarrow 0$ (with $\theta$  as per Remark~\ref{rmk:any-theta}) splitting the integral of $\Phi_{\alpha,0}$ into two parts: on $(0,\theta)$ monotone convergence applies; on $(\theta,\infty)$ -- dominated. If \eqref{eq:abelian} fails \eqref{eq:hitting-zero} follows at once; otherwise we get $\PP_x(\sigma_0<\infty)\geq e^{-\theta x}$, and since $\theta\in (0,\infty)$ is arbitrary, on letting $\theta\downarrow 0$, $\PP_x(\sigma_0<\infty)=1$.
\end{proof}

\begin{corollary}
Assume  $\Psi_m^{-1}(0)\geq \Psi_b^{-1}(0)>0$. Then $\PP_x(\sigma_0<\infty)<1$ for all $x\in (0,\infty)$. 
\end{corollary}
\begin{proof}
Take $\alpha=\bar\alpha=a=0$ in \eqref{eq:firs-passage.cbic}.
\end{proof}

\begin{remark}\label{rmk:not-cm}
In the context of Theorem~\ref{entrancetimeCBIculling} restrict to $\bar\alpha=0$ and $\alpha>0$. \eqref{eq:Phi}-\eqref{eq:firs-passage.cbic} were given under the condition $\Psi_m(\Psi_b^{-1}(0))\leq \alpha$, which is a little mysterious. To make it somewhat less so, we shall argue that when $\Psi_m(\Psi_b^{-1}(0))>0$, i.e. $ \Psi_m^{-1}(0)<\Psi_b^{-1}(0)$, then for $\alpha\in (0,\Psi_m(\Psi_b^{-1}(0)))$ (which implies $\Psi_m^{-1}(\alpha)<\Psi_b^{-1}(0)$) the identity $$\Phi_\alpha(x):=\PP_x[e^{-\alpha \sigma_0};\sigma_0<\mathsf{l}]=\int e^{-xz}\nu (\dd z),\quad x\in [0,\infty),$$ can hold for no measure $\nu$ on $\mathcal{B}_{[0,\infty)}$. Suppose per absurdum that it were so. Then for any starting value $x\in (0,\infty)$, taking $\mu=\delta_x$ in Theorem~\ref{prop:cbic-sde-method}, the process $(\Phi_\alpha(Y_t)e^{-\alpha(t\land \tau_0)}\mathbbm{1}_{\{t\land \tau_0<\zeta\}})_{t\in [0,\infty)}$ would be a bounded martingale (essentially by the Markov property; we refer to \cite[Proposition~3.1]{vidmar2021characterizations} for the detailed computation) and so by Theorem~\ref{prop:cbic-sde-method}\ref{prop:cbic-sde-method:mtg} and Remark~\ref{remark:mtg}
we would have $\AA\Phi_\alpha=\alpha\Phi_\alpha$ on $(0,\infty)$, i.e. 
$$\int( x\Psi_b(z)+\Psi_m(z))e^{-xz}\nu(\dd z)=\alpha\int e^{-xz}\nu(\dd z),\quad x\in (0,\infty).$$
Rearranging and effecting a Tonelli-Fubini gives 
$$\int e^{-xz}\Psi_b(z)\nu(\dd z)=\int_0^\infty e^{-xz}\int_{[0,z]}(\alpha-\Psi_m(y))\nu(\dd y)\dd z,\quad x\in (0,\infty);$$
therefore 
$$\Psi_b(z)\nu(\dd z)=\dd z\int_{[0,z]}(\alpha-\Psi_m(y))\nu(\dd y),\quad z\in [0,\infty).$$ Comparing the sign of the measure on the l.h.s. with the sign of the measure on the r.h.s. we see that $\nu$ must be carried by $[\Psi_m^{-1}(\alpha),\infty)$. In that case, since $\nu\ne 0$, we get that the r.h.s. has unbounded support and is a negative measure unless it is carried by $\Psi_m^{-1}(\{\alpha\})=\{\Psi_m^{-1}(\alpha)\}$ ($\because $ $\alpha>0$) in which case the r.h.s. is the zero measure. The former case is in contradiction with the fact that the l.h.s. is nonnegative on $[\Psi_b^{-1}(0),\infty)$. The latter case requires that $\nu$ also be carried by $\Psi_b^{-1}(\{0\})$, therefore $\Psi_b(\Psi_m^{-1}(\alpha))=0$, which is in contradiction with $\Psi_b^{-1}(0)>\Psi_m^{-1}(\alpha)$.

\noindent It means that we were not being ``completely dumb'' by failing,  for  $\alpha\in  (0,\Psi_m(\Psi_b^{-1}(0)))$, to recognize a different would-be (nonnegative) ``density kernel'' in \eqref{eq:Phi}; the completely monotone  \cite[Definition~1.3]{bernstein} character of $(0,\infty)\ni x\mapsto \PP_x[e^{-\alpha \sigma_0};\sigma_0<\mathsf{l}]$, being true for $\alpha\in [\Psi_m(\Psi_b^{-1}(0)),\infty)$, in fact does not extend to $\alpha\in  (0,\Psi_m(\Psi_b^{-1}(0)))$ (by Bernstein's theorem \cite[Theorem~1.4]{bernstein}). 
\end{remark}

We turn to explosions. 

\begin{theorem}[Explosion times of CBMs]\label{thm:explosions}
Let $\alpha,\bar\alpha$ be from $[0,\infty)$. Suppose $\Psi_m(\Psi_b^{-1}(\bar \alpha))<\alpha$.   If $\bar\alpha=0$ we assume further that $\int_{0+}\vert \Psi_b\vert^{-1}<\infty$ (explosivity condition for the associated CSBP with branching mechanism $\Psi_b$ \cite[Theorem~12.3]{Kyprianoubook}, in particular necessarily then $\Psi_b^{-1}(0)>0$).  Put 
\begin{equation}\label{eq:scale-for-explosion}
\Psi_{\alpha,\bar\alpha}(x):=1-\alpha Z_{\alpha,\bar\alpha}(x):=1-\alpha\int_0^{\Psi_b^{-1}(\bar{\alpha})}\frac{\dd z}{\bar{\alpha}-\Psi_b(z)}\exp\left(-xz-\int_0^z\frac{\alpha-\Psi_m(u)}{\bar\alpha-\Psi_b(u)}\dd u\right), \quad x\in [0,\infty).
\end{equation}
 Then for   $a\leq x$  from $[0,\infty)$:
\begin{enumerate}[(i)] 
\item when $\int_{0+}\vert \Psi_b\vert^{-1}<\infty$ (taking $\bar\alpha=0$),
\begin{equation}
\label{eq:explosion.cbic}
\PP_x[e^{-\alpha \mathsf{l}};\mathsf{l}<\sigma_a]=\Psi_{\alpha,0}(x)-\frac{\Phi_{\alpha,0}(x)}{\Phi_{\alpha,0}(a)}\Psi_{\alpha,0}(a);
\end{equation}
\item  for $\bar\alpha>0$, 
\begin{equation}\label{eq:end-times}
\PP_x\left[\int_0^{\sigma_a\land \mathsf{l}}e^{-\alpha s-\bar\alpha\int_0^sY_u\dd u}\dd s\right]=Z_{\alpha,\bar\alpha}(x)-\frac{\Phi_{\alpha,\bar\alpha}(x)}{\Phi_{\alpha,\bar\alpha}(a)}Z_{\alpha,\bar\alpha}(a).
\end{equation}
\end{enumerate}
\end{theorem}
Like in Theorem~\ref{entrancetimeCBIculling} the case $\bar\alpha=0$ is the one that is mainly of interest here. Though, \eqref{eq:end-times} has the following interpretation. The quantity $$\alpha\PP_x\left[\int_0^{\sigma_a\land \mathsf{l}}e^{-\alpha s-\bar\alpha\int_0^sY_u\dd u}\dd s\right]=\int_0^\infty \alpha e^{-\alpha s}\PP_x[e^{-\bar\alpha\int_0^sY_u\dd u}; s<\sigma_a\land \mathsf{l}]\dd s$$ is the probability that a CBM starting from $x$, and that is killed independently at rate $\alpha$, neither has reached $a$ nor has exploded nor has its running lifetime-to-date process $\int_0^\cdot Y_u\dd u$ exceeded an independent exponential random random variable of rate $\bar\alpha$,  before it is killed. In the Lamperti transform it corresponds to, ceteris paribus, $X$ being killed at rate $\alpha$, $L$ being killed at rate $\bar\alpha$, and then asking for the probability that starting from $x$, the process $Y$ is killed by $X$ before it has had a chance to be killed by $L$, to reach $a$, or to explode.

\begin{proof}
$\Psi_{\alpha,\bar\alpha}:[0,\infty)\to \mathbb{R}$ is well-defined, finite, $\uparrow\uparrow$, continuous with limit $1$ at infinity, which follows easily from the assumptions made. These properties in turn are mirrored in those of $Z_{\alpha,\bar\alpha}$. Next, we check that 
$$\mathcal{A}\Psi_{\alpha,\bar\alpha}(x)=\alpha\Psi_{\alpha,\bar\alpha}(x)+\bar \alpha x\Psi_{\alpha,\bar\alpha}(x)-\bar\alpha x,\quad x\in (0,\infty),$$
which again is just straightforward computation: 
\begin{align*}
\mathcal{A}\Phi_{\alpha,\bar\alpha}(x)=&\mathcal{L}^{\Psi_m}\Psi_{\alpha,\bar\alpha}(x)+x\mathcal{L}^{\Psi_b}\Psi_{\alpha,\bar\alpha}(x)\\
=&-\alpha\int_0^{\Psi_b^{-1}(\bar{\alpha})}\frac{\dd z}{\bar{\alpha}-\Psi_b(z)}\exp\left(-xz-\int_0^z\frac{\alpha-\Psi_m(u)}{\bar\alpha-\Psi_b(u)}\dd u\right)\left(\Psi_m(z)+x\Psi_b(z)\right)\\
=&\alpha\Psi_{\alpha,\bar\alpha}(x)+\bar \alpha x\Psi_{\alpha,\bar\alpha}(x)-\alpha-\bar\alpha x+\alpha \int_0^{\Psi_b^{-1}(\bar{\alpha})}\dd z \frac{\alpha-\Psi_m(z)}{\bar{\alpha}-\Psi_b(z)}\exp\left(-xz-\int_0^z\frac{\alpha-\Psi_m(u)}{\bar\alpha-\Psi_b(u)}\dd u\right)\\
&+\alpha x\int_0^{\Psi_b^{-1}(\bar{\alpha})}\dd z\exp\left(-xz-\int_0^z\frac{\alpha-\Psi_m(u)}{\bar\alpha-\Psi_b(u)}\dd u\right)\\
=&\alpha\Psi_{\alpha,\bar\alpha}(x)+\bar \alpha x\Psi_{\alpha,\bar\alpha}(x)-\alpha-\bar\alpha x-\alpha\left[\exp\left(-xz-\int_0^z\frac{\alpha-\Psi_m(u)}{\bar\alpha-\Psi_b(u)}\dd u\right)\big \vert_{z=0}^{\Psi_b^{-1}(\bar\alpha)}\right]\\
=&\alpha\Psi_{\alpha,\bar\alpha}(x)+\bar \alpha x\Psi_{\alpha,\bar\alpha}(x)-\bar\alpha x.
\end{align*}
It follows from Theorem~\ref{prop:cbic-sde-method}\ref{prop:cbic-sde-method:mtg}  \& Remark~\ref{remark:mtg} that for any $a\in (0,\infty)$, the  process
 \begin{align*}
M_t&:= \Psi_{\alpha,\bar\alpha}(Y(t))e^{-\alpha t-\bar\alpha\int_0^tY_s\dd s}-\Psi_{\alpha,\bar\alpha}(x)+\bar\alpha\int_0^{t}e^{-\alpha s-\bar\alpha\int_0^sY_u\dd u}Y_s\dd s,\\
&=-\alpha\left(Z_{\alpha,\bar\alpha}(Y(t))e^{-\alpha t-\bar\alpha\int_0^tY_s\dd s}-Z_{\alpha,\bar\alpha}(x)+\int_0^{t}e^{-\alpha s-\bar\alpha\int_0^sY_u\dd u}\dd s\right),\quad t\in [0,\zeta),
\end{align*}
 stopped at $\tau_a$, is a local martingale on $[0,\zeta)$ under $\PP^{\delta_x}$, $x\in [a,\infty)$. 
 
Setting $\bar\alpha=0$ we get because martingales have a constant expectation (the first form of $M$ is the most convenient, exploting $\lim_\infty\Psi_{\alpha,0}=1$), and from \eqref{eq:firs-passage.cbic}, the $\PP_x$-Laplace transform for the explosion time $\mathsf{l}$ on $\{\mathsf{l}<\sigma_a\}$, $x\in [a,\infty)$, $a\in (0,\infty)$. Letting $a\downarrow 0$ gives \eqref{eq:explosion.cbic} also for the case $a=0$. 
 
For $\bar\alpha>0$, we get similarly \eqref{eq:end-times} (but now the second form of $M$ appears to be more handy, using $\lim_\infty Z_{\alpha,\bar\alpha}=0$).
\end{proof}
\begin{corollary}\label{corollary:explosivity-refined}
The CSBP with branching mechanism $\Psi_b$ is explosive (equivalently, $\int_{0+}\vert \Psi_b\vert^{-1}<\infty$ \cite[Theorem~12.3]{Kyprianoubook}) iff the  CBM process $Y$ is explosive (i.e. $\PP^\mu(\zeta<\infty)>0$ for some, equivalently all initial distributions $\mu$ that are not concentrated at $0$).
\end{corollary}
\begin{proof}
We already know that if for some initial distribution $\mu$ (that is not concentrated at $0$), $\PP^\mu(\zeta<\infty)>0$, then the CSBP with branching mechanism $\Psi_b$ is explosive (Corollary~\ref{corollary:explosivity-I}). Now suppose the latter, i.e.  $\int_{0+}\vert \Psi_b\vert^{-1}<\infty$. 
Taking  $a=0$ in \eqref{eq:explosion.cbic} (with an arbitrary $\alpha\in (\Psi_m(\Psi_b^{-1}(\bar \alpha))\lor 0,\infty)$) we get $\PP_x(\mathsf{l}<\infty)>0$ for $x\in (0,\infty)$ (just because $\Psi_{\alpha,0}$ is $\uparrow\uparrow$, while $\Phi_{\alpha,0}$ is $\downarrow\downarrow$).
\end{proof}

\bibliographystyle{plain}
\bibliography{doku}
\end{document}